\documentclass[11pt]{amsart}

\usepackage{amssymb,amsfonts,latexsym,amsmath,mathrsfs,enumerate}

\usepackage{amssymb,amsfonts,latexsym,amsmath,mathrsfs,enumerate}
\usepackage{amsfonts}
\usepackage{latexsym,marvosym}

\DeclareMathOperator{\dist}{dist}
\DeclareMathOperator{\diam}{diam}
\DeclareMathOperator{\supp}{supp}

\newcommand{\eps}{\varepsilon}

\renewcommand{\phi}{\varphi}

\newcommand{\C}{\mathbb C}
\newcommand{\N}{\mathbb N}
\newcommand{\R}{\mathbb R}

\renewcommand{\P}{\mathbb P}

\newcommand{\EE}{\mathscr E}
\newcommand{\DD}{\mathscr D}
\newcommand{\CC}{\mathscr C}
\newcommand{\JJ}{\mathscr J}

\newcommand{\sfrac}[2]{{\textstyle \frac{#1}{#2}}}
\newcommand{\case}[4]{\left\{\begin{array}{ll}#1,&#2\\#3,&#4\end{array}\right.}

\newcommand{\lqq}{``}
\newcommand{\rqq}{''}

\newcommand{\<}{\langle}
\renewcommand{\>}{\rangle}

\newtheorem{te}{Theorem}
\newtheorem*{mainte}{Main theorem}
\newtheorem{pr}[te]{Proposition}
\newtheorem{lem}[te]{Lemma}

\newtheorem{co}[te]{Corollary}

\title{Whitney extension operators without loss of derivatives}

\author{Leonhard Frerick}
\address{Fachbereich IV -- Mathematik, Universit\"{a}t Trier, D-54286 Trier, Germany}
\email{frerick@uni-trier.de}

\author{Enrique Jord\'a}

\address{Departamento de Matem\'atica Aplicada, E. Polit\'{e}cnica Superior
de Alcoy, Universidad Polit\'ecnica de Valencia, Plaza Ferr\'andiz
y Carbonell 2, E-03801 Alcoy (Alicante), Spain}
\email{ejorda@mat.upv.es}

\author{Jochen Wengenroth}
\address{Fachbereich IV -- Mathematik, Universit\"{a}t Trier, D-54286 Trier, Germany}
\email{wengenroth@uni-trier.de}


\begin{document}

\begin{abstract}
For a compact set $K\subseteq \R^d$ we characterize  the
existence of a linear extension operator
$E:\EE(K)\to\CC^\infty(\R^d)$ for the space of Whitney jets $\EE(K)$
without loss of derivatives, that is, it satisfies the best possible
continuity estimates 
\[
\sup\{|\partial^\alpha
E(f)(x)|:  |\alpha|\le n, x\in\R^d\} \le C_n \|f\|_{n},
\]
where $\|\cdot\|_n$ denotes the $n$-th Whitney norm. The
characterization is a surprisingly simple purely geometric condition telling
in a way that at all its points, $K$ is big enough in all directions.
\end{abstract}

\maketitle

\section{The main result}

The problem that compact sets  $K\subseteq \R^d$ are often too small to determine all derivatives of differentiable functions on it was
overcome by Whitney's ingeneous invention of spaces $\EE^n(K)$ and $\EE(K)=\EE^\infty(K)$  of {\em jets} (of
finite and infinite order, respectively) which he proved to be exactly the spaces of restrictions $(\partial^\alpha f|_K)_\alpha$ for $f\in \CC^n(\R^d)$, 
$n\in\N\cup\{\infty\}$.
In the finite order case, the extension can be even done by a continuous linear operator, and the last eighty years have seen countless
results about the notoriously difficult problem to characterize the existence of continuous linear extension operators $\EE(K)\to\CC^\infty(\R^d)$ in the infinite order case
if both spaces are endowed with their natural families of norms.

In the present article we charaterize the existence of operators which, simultaneously for all $n\in\N_0\cup\{\infty\}$, are extensions $\EE^n(K)\to\CC^n(\R^d)$.
Till now, only very few cases were understood, the most prominent result being Stein's extension operator for sets with Lipschitz boundary.
In view of the apparent difficulty of the unrestricted case and to our own surprise the final answer for the case of 
{\em extension operators without loss of derivatives} is
strikingly simple:

\begin{mainte} 
A compact set $K\subseteq \R^d$ has an extension operator $\EE(K)\to \CC^\infty(\R^d)$  without
loss of derivatives if and only if there is $\varrho \in (0,1)$ such that, for every $x_0\in K$ and $\eps\in (0,1)$,
there are $d$ points $x_1\ldots,x_d$ in $K\cap B(x_0,\eps)$ satisfying 
$ \dist(x_{n+1}, \text{\rm affine hull}\{x_0,\ldots,x_n\}) \ge \varrho \eps$  for all $n\in\{0,\ldots,d-1\}$.

\end{mainte}
That this geometric condition is  Bos and Milman's reformulation of a characterization of sets with the local Markov property due to 
Jonsson, Sj\"ogren, and Wallin is explained in the following section.  

\section{Introduction}

For a compact set $K\subseteq \R^d$ we denote by $\EE^n(K)$ and $\EE(K)=\EE^\infty(K)$ the spaces of Whitney jets of
finite and infinite order, respectively,
that is, families $f=\left(f^{(\alpha)}\right)_{|\alpha|<n+1}$ of continuous (real or complex valued)
 functions whose formal Taylor polynomials (for finite $n$)
\[
T_y^n(f)(x)=\sum\limits_{|\alpha|\le n}\frac{f^{(\alpha)}(y)}{\alpha!}(x-y)^\alpha
\]
 give the \lqq correct\rqq\ approximation as if $f^{(\alpha)}$ were the partial derivative
of order $\alpha$, namely, the local \lqq approximation errors\rqq
\[
q_{n,t}(f)=\sup\left\{\frac{|f^{(\alpha)}(x)- \partial^\alpha
T_y^{n}(f)(x)|}{|x-y|^{n-|\alpha|}}: |\alpha|\le n, x,y\in K, 0< |x-y|\le t\right\}
\]
tend to $0$ for $t\to 0$ (and all $n$ in the case of $\EE(K)$).
The $n$-th Whitney norm is then
\[
\|f\|_n= \sup\{|f^{(\alpha)}(x)|: x\in K,\, |\alpha|\le n\} + \sup\{q_{n,t}(f): t>0\}.
\]

Clearly, Taylor's theorem implies that for any  $n \in\N_0 \cup \{\infty\} $ and
$g\in\CC^{n}(\Omega)$ with an open set $\Omega$ containing $K$
the restrictions $\left(\partial^\alpha g|_K\right)_{|\alpha|<n+1}$ are
jets of order $n$. A
celebrated result of Whitney \cite{Wh} says that each jet $f\in\EE^n(K)$ 
has such an extension. He even
proved that one can extend jets of {\it finite} order by a continuous
linear operator. 

The existence of continuous linear extension operators for jets of infinite order depends on the shape
of the compact set and there is a vast amount of literature about this question, we refer to the introduction of
\cite{Fr} for an overview.					


The problem we deal with in this paper is the characterization of compact sets $K$ having an extension
operator $E:\EE(K)\to \CC^\infty(\R^d)$ which induces, simultaneously for all $n$, continuous extension operators
$E:\EE^n(K) \to \CC^n(\R^d)$. 
Since $\EE(K)$ is dense in $\EE^n(K)$ this is the equivalent to the continuity estimates 
\[
\sup\{|\partial^\alpha
E(f)(x)|:  |\alpha|\le n, x\in\R^d\} \le C_n \|f\|_{n} \text{ for all } n\in\N_0.
\]

Seeley \cite{Seeley} gave a simple construction of such operators for half spaces
and 
Stein \cite{St} found extension operators without loss of
derivatives for compact subsets $K$ with
$\mbox{Lip}_1$-boundary. Also Rogers \cite{Rogers} gave a sufficient geometric condition for $K$ permitting the
existence of an operator such that even all the Sobolev spaces $W_k^p(K)$ can be extended
to $W_k^p(\R^d)$ for each $k\in\N$ and $1\leq p\leq \infty$.

Our main result will use a (local) Markov inequality for polynomials 
in a form (apparently) first used by Jonsson, Sj\"ogren, and Wallin \cite{JSW}.  Paw{\l}ucki and
Ple\'sniak \cite{PaPl1,PaPl2} as well as Ple\'sniak \cite{Pl} used a global version to characterize
the existence of continuous linear extension operators for $\EE(K)$ with a weaker topology, and Bos and
Milman \cite{BoMi} introduced a {\em local Markov inequality} with some exponent $r \ge 1$ (LMI($r$)) on a compact set $K$
to obtain extension operators with homogeneous loss of differentiability. The exact loss of
differentiability was then charcterized in \cite{frerickjordawengenroth2010}.

Let us now give the precise definition.
$K\subseteq \R^d$ satisfies the LMI($r$)
if there exist $\eps_0>0$ and constants $c_k \ge 1$ such that for each polynomial $P$ of degree
$\deg(P)\le k$, each $\eps\in(0,\eps_0)$,
and each $x_0\in K$ we have
\[
|\nabla P(x_0)| \le c_k \eps^{-r} \|P\|_{B(x_0,\eps)\cap K}
\]
where $\|\cdot\|_M$ denotes the uniform norm on a set $M\subseteq \R^d$ and
$B(x_0,\eps)$ the closed ball of radius $\eps$
centered at $x_0$.

By applying the estimate $k$-times with $2^{-k}\eps$ instead of
$\eps$ we obtain (with different constants $c_k$) for all $\alpha\in\N_0^d$
\[
|\partial^\alpha P(x_0)|\le c_k \eps^{-r|\alpha|} \|P\|_{B(x_0,\eps)\cap K}
\]
which is LMI$(r)$ in the form used by Bos and Milman. 

Modifying the constants in the definition of LMI one can replace 
the existence quantifier for $\eps_0$  by the universal quantifier or just by $\eps_0=1$.
(Bos and Milman kept track of the constants and therefore, the formulation with $\eps_0$
was used, for our purpose, the constants are not important).

LMI measures the \lqq local size\rqq\ of $K$ near its boundary points and will serve as the characterization
of the existence of extension operators without loss of derivatives in our main contribution:

\begin{te}
\label{theorem6} A compact set $K\subseteq \R^d$ has an extension operator $\EE(K)\to \CC^\infty(\R^d)$  without
loss of derivatives if and only if $K$ satisfies {\rm LMI($1$)}.
\end{te}

The big advantage of LMI($1$) compared to the case $r>1$ is that 
Jonsson, Sj\"ogren, and Wallin \cite[Theorem 1.2 and 1.3] {JSW} showed that it is enough to check LMI($1$) for polynomials of degree $1$ 
which enabled them to obtain a purely geometric characterization figuring in the following corollary.

\begin{co}
A compact set $K\subseteq \R^d$ has an extension operator $\EE(K)\to \CC^\infty(\R^d)$  without
loss of derivatives if and only if there is $\varrho \in (0,1)$ such that, for every $x_0\in K$ and $\eps\in (0,1)$, $K\cap B(x_0,\eps)$
is not contained in any band of the form $\{x\in \R^d: |\langle b, x-x_0\rangle| \le \varrho \eps\}$ where $b\in\R^d$ is any vector of norm $1$.
\end{co}

An equivalent formulation was given by Bos and Milman in \cite[Theorem D]{BoMi}: $K\subseteq \R^d$ satisfies LMI($1$) if and only if there is $\varrho \in (0,1)$ such that
for every $x_0\in K$ and $\eps\in (0,1)$  there are $d$ points $x_1\ldots,x_d$ in $K\cap B(x_0,\eps)$ such that for all $n\in\{0,\ldots,d-1\}$
\[
 \dist(x_{n+1}, \text{affine hull}\{x_0,\ldots,x_n\}) \ge \varrho \eps
\]
(together with this reformulation, the corollary proves the main theorem).

These geometric conditions are very easy to check in concrete cases so that, for instance, our result includes Stein's theorem about compact sets with
Lip$_1$-boundary as well as sets with {\em inward} directed cusps (which are covered neither by Stein's theorem nor by Roger's results). Moreover, we easily obtain that
such porous sets like Cantor's or the Sierpi\'nski triangle admit extension operators without loss of differentiability whereas sets with outward directed cusps do not.

\medskip
Let us close this introduction with the remark that there is of course an alternative approach to smooth functions on small sets just
by setting $\CC^\infty(K)=\{f|_K: f\in \CC^\infty(\R^d)\}$ endowed with the quotient topology. If $\EE(K)$ admits a linear continuous
extension operator at all then, by a result in \cite{Fr}, $\EE(K)$ and $\CC^\infty(K)$ coincide and our result thus applies to the latter space.
However, if $K$ is too small, much less is known about extension operators for $\CC^\infty(K)$. On the one hand, there is a very deep result of
Bierstone and Milman \cite{BiMi} about semicoherent subanalytic sets and on the other hand, two special cases of non-subanalytic sets
were treated  by Fefferman and Ricci in \cite{FeRi}.

\medskip
The rest of the paper is organized as follows.
Since necessity of LMI($1$) follows from results in \cite{BoMi} and a short explicit proof is also contained in \cite{frerickjordawengenroth2010}, 
we only have to show sufficiency. 
In section 3 we will explain the construction of the extension opearator which is based on certain measures $\mu_\alpha$ so that
$\int f d\mu_\alpha$ interpolate the partial derivatives (this part is similar to our previous article \cite{frerickjordawengenroth2010}
and was inspired by Whitney's original construction). In the fourth section we will then show how to obtain those measures
if $K$ satisfies LMI($1$).

\section{Construction of the extension operators}

Let us first recall
 Whitney's explicit construction of an extension operator for $\EE^n(K)$. For a suitable partition of unity
$(\varphi_i)_{i\in\N}$ of $\Omega\setminus K$
(where $\Omega$ is an open set containing $K$)
such that the supports of $\varphi_i$ tend to $K$, and for $x_i\in K$ minimizing the
distance to $\supp(\varphi_i)$
the operator $E_n$ is of the following form
\[
E_n(f)(x)=\case{f^{(0)}(x)}{x\in K}{\sum\limits_{i=1}^\infty\varphi_i(x)T^n_{x_i}(f)(x)}{x\notin K}.
\]
As we want to have an operator which works simultaneously for all $n\in\N_0$ we replace the $n$-th degree Taylor polynomials around $x_i$ 
by certain \lqq interpolations\rqq\ which only depend on $f^{(0)}$, namely 
\[
 S_{i}(f)(x)= \sum_{|\alpha|\le i} \frac{1}{\alpha !}  \mu_{\alpha,i}(f^{(0)}) (x-x_i)^\alpha,
\]
where $\mu_{\alpha,i}(f^{(0)})=\int f^{(0)} d\mu_{\alpha,i}$ is the integral of $f$ with respect to a suitable (complex) measure.

\begin{te} \label{theorem 3}
Suppose that for all  $\alpha\in\N_0^d$, $x\in\partial K$, and $\eps>0$ there are measures $\nu_{\alpha,x,\eps}$ on $K$ such that, for all $n\in\N_0$ and  $f\in \EE^n(K)$,
\begin{align*}
&\lim_{\eps\to 0}\sup_{|\alpha|\leq n,x\in\partial K}\frac{|\nu_{\alpha,x,\eps}(f^{(0)})-\eps^{|\alpha|}f^{(\alpha)}(x)|}{\eps^n}=0
\text{ and }\\
&\lim_{\eps\to 0}\sup_{|\alpha|> n,x\in\partial K}\frac{|\nu_{\alpha,x,\eps}(f^{(0)})|}{\eps^n}=0.
\end{align*}
Then $K$ has an extension operator without loss of derivatives.
\end{te}

\begin{proof}
 
We consider the partition of unity  constructed by Whitney \cite{Wh}. For $K\subseteq \R^d$ compact there are an open set $\Omega$
containing $K$ and positive test functions $\varphi_i\in\DD(\Omega\setminus K)$ with the following properties.
\begin{enumerate}[(i)]
\item $\sum\limits_{i=1}^\infty \varphi_i(x) =1$ for all $x\in\Omega\setminus K$ and each point belongs to at most $N$ supports $\supp(\varphi_i)$
    for some constant $N\in\N$. 
\item $\supp(\varphi_i)\to K$ for $i\to\infty$, that is, for each $\eps>0$ there is $k\in\N$ such that
    $\supp(\varphi_i)\subseteq\{x\in\R^d: \dist(x,K)<\eps\}$ for all $i\ge k$. 
\item $\diam(\supp(\varphi_i)) \le 2\dist(\supp(\varphi_i),K)$
    (where $\diam$ is the diameter of a set). 
\item There are constants $c_\beta$ such that $\left|\partial^\beta\varphi_i(x)\right|\le c_\beta
    \dist(x,K)^{-|\beta|}$ for all $i\in\N$, $\beta\in\N_0^d$, and $x\in\R^d$.
\end{enumerate}

With the help of this partition Whitney showed that the operators $E_n$ defined above are indeed continuous linear extension operators from $\EE^n(K)$ to $\CC^n(\R^d)$.
Let us denote 
$$
\gamma_i =\dist(K,\supp(\varphi_i))=\dist(x_i,\supp(\varphi_i)).
$$
For $|\beta|\leq n$ and $|\alpha|\leq n$ we can use Leibniz's rule, (iii), and (iv)  to obtain positive constants $C_n$ independent of $i$ such that
 \begin{equation}
 \label{one1}
 \left|\partial^\beta\left((x-x_i)^\alpha\varphi_i(x)\right)\right|\leq C_n\gamma_{i}^{|\alpha|-|\beta|} \text{ for all } x\in \supp(\varphi_i).
 \end{equation}

 For $|\beta|\leq n$ and $|\alpha|>n$ we observe that property (iii) implies that $|x-x_i|\leq 3\gamma_i$ for all
 $x\in\supp(\varphi_i)$. Using Leibniz's rule again, we find (different) $C_n$ not depending on $i$ such that
 \begin{equation}
 \label{two2}
 \left|\partial^\beta\left((x-x_i)^\alpha\varphi_i(x)\right)\right|\leq C_n3^{|\alpha|}\sup_{\gamma\leq
 \beta,\gamma\leq\alpha}\frac{\alpha!}{(\alpha-\gamma)!}\gamma_{i}^{|\alpha|-|\beta|},\ x\in \supp(\varphi_i).
\end{equation}
We remark that for $|\beta|\leq n$
\begin{equation} \label{ola}
\sum_{|\alpha|>n}  \sup_{\gamma\leq \beta,\gamma\leq\alpha}\frac{1}{(\alpha-\gamma)!}3^{|\alpha|}
\leq 3^{|\beta|}\sum_{\gamma\leq\beta}\sum_{|\alpha|>n,\alpha\geq \gamma}\frac{1}{(\alpha-\gamma)!}3^{|\alpha-\gamma|}\leq
\end{equation}
$$
\leq 3^n (n+1)^d\left(\sum_{j\in\N_0}\frac{1}{j!}3^{j}\right)^d= e^{3d} (n+1)^d 3^n.
$$
We set  $\mu_{\alpha,i}=\nu_{\alpha,x_i,\gamma_i}/\gamma_i^{|\alpha|}$, and for $f \in \EE^0(K)=\CC(K)$ we define
\[
E(f)=\case{f^{(0)}(x)}{x\in K}{\sum\limits_{i\in \N}  \varphi_i(x)  \sum\limits_{|\alpha|\leq i}\frac{1}{\alpha!}\mu_{\alpha,i}(f^{(0)})(x-x_i)^{\alpha}}{x\notin K}.
\]
With Whitney's operators $E_n$  we will show below that for all $|\beta|\leq n$, $f\in\EE^n(K)$  we have

\begin{equation}
\label{closedgraph} \left|\partial^\beta E(f)(x)-\partial^\beta E_n(f)(x)\right|=o(\dist(x,K)^{n-|\beta|}) \quad \text{ for } x\to\partial K.
\end{equation}

This implies that $(E-E_n)(f)$ admits derivatives up to order $n$ in (the boundary of) $K$ and that they all 
vanish on $K$. Since the partition $\varphi_i$ is locally finite $E(f)$ is clearly $\CC^\infty$ on $\R^d\setminus K$. Thus  $E-E_n:\EE^n(K)\to \CC^n(\R^d)$ is a well
defined linear operator and takes its values in $\JJ^n(K)=\{g\in \CC^n(\R^d):\partial^\alpha g(x)=0\mbox{ for all } x\in K, \ |\alpha|\leq n\}$. 

It is clear that the operator is continuous if we consider in $\JJ^n(K)$ the topology of pointwise convergence  in
$\R^d\setminus K$. Since this topology is Hausdorff we can apply the closed graph theorem to conclude that $E-E_n:\EE^n(K)\to \JJ^n(K)$ is continuous
with respect to the Fr\'echet space topology on $\JJ^n(K)$, and therefore
also $E:\EE^n(K)\to \CC^n(\R^d)$ is continuous.

Let us now prove (\ref{closedgraph}). For $x\in \R^d\setminus K$ we define $i(x)=\min\{i\in I: x\in\supp \varphi_i\}$.
Because of the property (ii) we then have $i(x)\to \infty$ if $x\to \partial K$. 
For $|\beta|\leq n$, $f\in \EE^n(K)$, and $i(x)>n$  we have 
\begin{align*}
&\partial^\beta(E(f)-E_n(f))(x)  = &  \\
 & \sum_{i \ge i(x)}\sum_{|\alpha|\leq n}\frac{1}{\alpha!} (\mu_{\alpha,i}(f^{(0)})- f^{(\alpha)}(x_i))\partial^\beta((x-x_i)^\alpha\varphi_i(x)) \\
+ & \sum_{i \ge i(x)}\sum_{n<|\alpha|\leq i}\frac{1}{\alpha!} \mu_{\alpha,i}(f^{(0)})\partial^\beta((x-x_i)^\alpha\varphi_i(x)).
\end{align*}

We will estimate both terms. Using the hypotheses on the measures we get for $|\alpha|\leq n$
$$
|\mu_{\alpha,i}(f^{(0)})- f^{(\alpha)}(x_i)|=o(\gamma_{i}^{n-|\alpha|}) \quad\mbox{as }i \to\infty.
$$
 From this, (\ref{one1}), and the bound for the number of supports that can contain $x$ we obtain
$$
\left|\sum_{|\alpha|\leq n}\frac{1}{\alpha!} (\mu_{\alpha,i}(f^{(0)})-f^{(\alpha)}(x_i))\partial^\beta((x-x_i)^\alpha\varphi_i(x))\right|
=o(\gamma_i^{n-|\beta|})\text{ for } i\to \infty,
$$
where the limit is uniform with respect to $x\in \R^d$.

From $\gamma_i\leq d(x,K)\leq 3\gamma_i$ for each $x\in \supp(\varphi_i)$ it then follows that
$$
\lim_{x\to \partial K}\sum_{i \ge i(x)}\frac{1}{d(x,K)^{n-|\beta|}}\left|\sum_{|\alpha|\leq n}\frac{1}{\alpha!} 
(\mu_{\alpha,i}(f^{(0)})-f^{(\alpha)}(x_i))\partial^\beta((x-x_i)^\alpha\varphi_i(x))\right|
=0.
$$
For the second summand we use again the hypotheses to get
$$
\sup_{|\alpha|>n}\left|\mu_{\alpha,i}(f^{(0)}) {\gamma_i^{|\alpha|}}\right|=o(\gamma_{i}^{n})\quad\mbox{as } i \to\infty.
$$
Then (\ref{two2}) and (\ref{ola}) imply that  (again uniformly in $x$) 
\begin{align*}
& \left|\sum_{|\alpha|> n}\frac{1}{\alpha!} \mu_{\alpha,i}(f^{(0)})\partial^\beta \left((x-x_i)^\alpha\varphi_i(x)\right)\right| \\ 
\leq &
  \sum_{|\alpha| > n} \frac{1}{\alpha!} \gamma_i^{-|\alpha|} o (\gamma_i^n) C_n \sup_{\gamma\leq\beta,\gamma\leq\alpha}
   \frac{\alpha!}{(\alpha-\gamma)!}3^{|\alpha|} \gamma_{i}^{|\alpha|-|\beta|} \\
\leq & 
 o(\gamma_i^n) C_n e^{3d} (n+1)^d 3^n \gamma_{i}^{-|\beta|} = o(\gamma_{i}^{n-|\beta|}) \quad \mbox{ as } i\to \infty.
\end{align*}

Altogether we obtain (remembering $i(x)\to\infty$ if $x\to \partial K$)
$$
\lim_{x\to \partial K}
\sum_{i\ge i(x)}\frac{1}{d(x,K)^{n-|\beta|}}\left|\sum_{n<|\alpha|\leq i}\frac{1}{\alpha!} \mu_{\alpha,i}(f^{(0)})\partial^\beta((x-x_i)^\alpha\varphi_i(x))\right|=0$$
which gives (\ref{closedgraph}).
\end{proof}

\section{Construction of the measures}

To finish the proof of the main theorem we have to show the existence of the measures figuring in theorem \ref{theorem 3}:

\begin{pr}
\label{measures}
Let $K\subseteq \R^d$ be a compact subset satisfying {\rm LMI($1$)}.
For all $\alpha\in\N_0^d$, $x\in \partial K$, and $\eps \in (0,1)$ there is a measure $\nu_{\alpha,x,\eps}$ on $K$ such that, for all  $n\in\N_0$ and $f\in \EE^n(K)$,
\begin{align*}
&\lim_{\eps\to 0}\sup_{|\alpha|\leq n,x\in\partial K}\frac{|\nu_{\alpha,x,\eps}(f^{(0)})-\eps^{|\alpha|} f^{(\alpha)}(x)|}{\eps^n}=0
\text{ and }\\
&\lim_{\eps\to 0}\sup_{|\alpha|> n,x\in\partial K}\frac{|\nu_{\alpha,x,\eps}(f^{(0)})|}{\eps^n}=0.
\end{align*}
\end{pr}

The main ingredient in the proof will be the solution of a suitable moment problem and in order to get uniform estimates for $x\in \partial K$ we will apply
scaling arguments. It is therefore convenient to consider the following \lqq blow-ups\rqq\ of $K$ with respect to a boundary point:
\[
 A_{x,\eps}= \eps^{-1} (K-x) \cup \{y\in\R^d: |y|\ge \eps^{-1}\}.
\]
The union with the complement of the large ball is necessary to solve the moment problem (remember that such problems behave quite differently on bounded and unbounded sets).
However, cutting off the part of the measures supported in $\{|y|\ge \eps^{-1}\}$ will have no influence on the properties required in theorem \ref{theorem 3}.

The moment problem is described in the following proposition:

\begin{pr}
\label{momentproblem}
Let $K\subseteq \R^d$ be a compact subset satisfying {\rm LMI($1$)}. Then there exists a continuous and radial function $\varrho:\R^d\to(0,\infty)$ with
$|y|^n =o(\varrho(y))$ for $|y|\to\infty$ and all $n\in\N$ such that for
each $x\in \partial K$, $\eps \in (0,1)$, and $\alpha\in \N_0^d$ there exists a finite regular Borel measure $\mu:=\mu_{\alpha,x,\eps}$ on $A_{x,\eps}$ with
total variation
$|\mu|(A_{x,\eps})\leq 1$ such that
\[
\int_{A_{x,\eps}}y^{\beta}\frac{1}{\varrho(y)}d\mu(y)= \case{\alpha !}{\beta=\alpha}{0}{\text{else}}.
\]
\end{pr}

Before proving this let us show how to obtain Proposition \ref{measures}.

\begin{proof}[Proof of Proposition \ref{measures}]

We assume without loss of generality that $K \subseteq B(0,\frac{1}{4})$.  For $x\in \partial K$, $\eps \in (0,1)$, $\alpha\in \N_0^d$, and $f\in \CC(K)$  we define
$$
\nu_{\alpha,x,\eps}(f)=\int_{\eps^{-1}(K-x)}f(\eps y+x)\frac{1}{\varrho(y)}d\mu_{\alpha,x,\eps}(y),
$$
with $\mu_{\alpha,x,\eps}$ from proposition \ref{momentproblem}.
For each $f\in \CC(\R^d)$ with support in $B(0,3/4)$ we then have
\[
 \nu_{\alpha,x,\eps}(f|_K)= \int_{A_{x,\eps}}f(\eps y+x)\frac{1}{\varrho(y)}d\mu_{\alpha,x,\eps}(y)
\]
since  $|x+\eps y|> 3/4$ whenever $|y|> 1/\eps$. 

Mulitplying with a cut-off function we may assume that Whitney's extension operator $E_n:\EE^n(K)\to \CC^n(\R^d)$ 
has values in the space of $\CC^n$-functions with support in $B(0,3/4)$.
For $f\in \EE^n(K)$ we denote by $F=E_n(f)$ an extension of $f$ and obtain from  
Taylor's theorem and the condition on the moments of $\mu_{\alpha,x,\eps}$
$$
| \nu_{\alpha,x,\eps}(f^{(0)})-\eps^{|\alpha|}f^{(\alpha)}(x)|=
\left|\int_{A_{x,\eps}}F(\eps y+x)\frac{1}{\varrho(y)}d\mu_{\alpha,x,\eps}(y)-\eps^{|\alpha|}\partial^{\alpha}F(x)\right|=
$$
$$
=\left|\int_{A_{x,\eps}}\left(\sum_{|\gamma|<n}\frac{\partial^\gamma F(x)}{\gamma!}\eps^{|\gamma|}y^{\gamma}+\eps^n\sum_{|\gamma|=n}
    \frac{\partial^\gamma F(\xi)}{\gamma!}\eps^ny^{\gamma}\right)\frac{1}{\varrho(y)}d\mu_{\alpha,x,\eps}(y)\right.
$$
$$
\quad \left.-\eps^{|\alpha|}\partial^{\alpha}F(x)\right|=
  \eps^n\left|\int_{A_{x,\eps}}\sum_{|\gamma|=n}\left(\frac{\partial^\gamma F(\xi)}{\gamma!}-\frac{\partial^\gamma F(x)}{\gamma!}\right)
   \frac{y^\gamma}{\varrho(y)}d\mu_{\alpha,x,\eps}(y)\right|
$$
with $\xi:=\xi(x,\eps,y,f)\in [x,x+\eps y]$. We split the integral into the parts over $A_{x,\eps} \cap \{|y|\le r\}$ and $A_{x,\eps} \cap \{|y|> r\}$.
The first integral then becomes small (for $\eps \to 0$ and each fixed $r$) because of the uniform continuity of $\partial^\gamma F$ and the second becomes small 
(for $r\to\infty$ uniformly in $\eps$) because of the boundedness of
$\partial^\gamma F$, $|\mu_{\alpha,x,\eps}|(A_{x,\eps}) \le 1$, and $|y|^{|\gamma|}/\varrho(y) \to 0$ for $|y|\to\infty$.

For $|\alpha|>n$ we  compute similarly
$$
\sup_{|\alpha|>n,x\in \partial K}\frac{|\nu_{\alpha,x,\eps}(f^{(0)})|}{\eps^n}
=\sup_{|\alpha|>n,x\in \partial K}\left|\frac{1}{\eps^n}\int_{A_{x,\eps}}F(\eps y+x)\frac{1}{\varrho(y)}d\mu_{\alpha,x,\eps}(y)\right|
$$
$$
=\sup_{|\alpha|>n,x\in \partial K}\left|\int_{A_{x,\eps}}\sum_{|\gamma|=n}\left(\frac{\partial^\gamma F(\xi)}{\gamma!}-
\frac{\partial^\gamma F(x)}{\gamma!}\right)\frac{y^{\gamma}}{\varrho(y)}d\mu_{\alpha,x,\eps}(y)\right|,
$$
with $\xi:=\xi(x,\eps,y,f)\in [x,x+\eps y]$. By the same arguments as above we get
$$
\lim_{\eps\to 0}\sup_{|\alpha|>n,x\in \partial K}\frac{|\nu_{\alpha,x,\eps}(f^{(0)})|}{\eps^n}=0.
$$

\end{proof}

The proof of proposition \ref{momentproblem} will use duality where an improved version of 
LMI($1$) implies a suitable continuity estimate. For this improvement we first show a very simple lemma about the Markov inequality for balls.
We denote by $\P_k$ the space of polynomials of degree less or equal $k$.

\begin{lem}
\label{rhobigger2}
For each $k\in\N$ there is $C_k>0$ such that for all $Q\in\P_k$ and all $\varrho\geq 2$ we have
$$
\sum_{|\alpha|\leq k}\frac{|\partial^{\alpha}Q(x)|}{\alpha!}\varrho^{|\alpha|}\leq C_k \sup_{1\leq |y-x|\leq \varrho}|Q(y)|.
$$
\end{lem}

\begin{proof}
It is enough to show the statement for $x=0$. Since both sides of the inequality are norms on the finite dimensional space $\P_k$, there are constants $C_k$ such that the inequality
holds for $\varrho=2$. For $\varrho>2$ we denote $Q_\varrho(x)=Q(\sfrac{\varrho}{2} x)$ and obtain
\begin{align*}
\sum_{|\alpha|\leq k}\left|\frac{\partial^{\alpha}Q(0)}{\alpha!}\right|\varrho^{|\alpha|}
& =\sum_{|\alpha|\leq k}\left|\frac{\partial^{\alpha}Q_\varrho(0)}{\alpha!}\right|2^{|\alpha|}
\leq C_k\sup_{1\leq |y|\leq 2}|Q_\varrho(y))|\\
& =
C_k\sup_{\frac{\varrho}{2}\leq|y|\leq \varrho}|Q(y)|\leq C_k\sup_{1\leq|y|\leq \varrho}|Q(y)|.
\end{align*}
\end{proof}

The next lemma gives a uniform Markov inequality for the highest degree derivatives.

\begin{lem}
\label{bound}
Let $K\subseteq \R^d$ be a compact subset satisfying {\rm LMI($1$)}. For all $k\in\N$ there is $C_k> 0$ such that for all $\eps \in (0,1)$, $x\in \partial K$, 
$r> 0$ and $P\in \P_k$ we have
$$
\sum_{|\alpha|= k}\left|\frac{\partial^{\alpha}P(0)}{\alpha!}\right|\leq\frac{C_k}{r^k} \sup\{|P(y)|:y\in A_{x,\eps},|y|\leq r\}.
$$
\end{lem}

\begin{proof}
We may assume $\eps_0=2$ in the definition of the Markov inequality and take the sequence $c_k$ from there.
Fix $k\in\N$, $x\in \partial K$, $0<\eps\leq 1$, $P\in \P_k$, and $0\leq r\leq\frac2\eps$. Set $Q(y):=P(\frac{1}{\eps}(y-x))$. 
For suitable constants $C_k$(depending on $c_k$ and $k$) we get
\begin{align*}
\sum_{|\alpha|= k}\left|\frac{\partial^{\alpha}P(0)}{\alpha!}\right|r^{k}
& =  \sum_{|\alpha|= k}\left|\frac{\partial^{\alpha}Q(x)}{\alpha!}\right|r^{k}\eps^k
\leq C_k\sup_{y\in B(x,r\eps)\cap K}|Q(y)| \\
& =  C_k \sup\left\{|P(z)|:\ z\in \eps^{-1} \left(K-x\right)\cap B(0,r)\right\}
\end{align*}
If on the other hand $r\geq \frac{2}{\eps}$ we set $\varrho=r\eps$ and we get from lemma \ref{rhobigger2} (with different constants)
\begin{align*}
& \sum_{|\alpha|= k}\left|\frac{\partial^{\alpha}P(0)}{\alpha!}\right|r^{k} 
 \le \sum_{|\alpha|\le k}\left|\frac{\partial^{\alpha}Q(x)}{\alpha!}\right|\varrho^k 
 \le C_k \{\sup |Q(y)|:1\leq |y-x|\leq \varrho \}
\\ & = 
C_k \sup\{|P(z)|:\eps^{-1} \leq |z| \leq r\} \le C_k \sup \{|P(z)|: z\in A_{\eps,x},|z|\leq r\}.
\end{align*}
\end{proof}

The following lemma is the main technical tool to solve the moment problem in proposition \ref{momentproblem}:

\begin{lem}
\label{weight}
Let $K\subseteq \R^d$ be a compact subset satisfying {\rm LMI($1$)}. For each family $(\eps_\alpha)_{\alpha\in\N_0^d}$ of positive numbers 
there is a continuous and radial function $\varrho:\R^d\to (0,\infty)$  with $|y|^k=o(\varrho(y))$ for $y\to\infty$ 
and each $k\in\N$ such that for all $x\in\partial K$, $\eps\in (0,1)$, and all polynomials $P$ we have
$$
\sup_{y\in A_{x,\eps}}\frac{|P(y)|}{\varrho(y)}\leq 1
\quad \Longrightarrow \quad
\frac{|\partial^{\alpha}P(0)|}{\alpha!}<\eps_\alpha \text{ for all } \alpha\in\N_0^d.
$$
\end{lem}

\begin{proof}
Let $(\eps_\alpha)_{\alpha\in\N_0^d}$ be given. For a suitable increasing sequence $R_k\to\infty$ we define a radial weight function $\varrho:\R^d\to (0,\infty)$  by
$\varrho(0)= \eps_0/2$ and
\[
 \varrho(x)= |x|^{k-1} \text{ for } R_{k-1} < |x|\le R_k.
\]
We will construct the sequence $R_k$ so that $\varrho$ satisfies the assertion of the lemma. Afterwards it is easy to find a continuous modification.

Proceeding recursively, we will show that $(R_k)_k$ can be found 
such that for all polynomials $P \in \P_k$, $x\in \partial K$, and $\eps \in (0,1)$ we have
\[
\sup_{y\in A_{x,\eps},|y|\leq R_k}\frac{|P(y)|}{\varrho(y)}\leq 1
\quad \Longrightarrow \quad
\frac{|\partial^{\alpha}P(0)|}{\alpha!}< \eps_\alpha
\]
for all $|\alpha|\leq k$. Since $\varrho(0)<\eps_0$ this is true for $k=0$. We assume that $R_0,\ldots,R_{k-1}$ are constructed.

By a compactness argument (for the unit ball of the finite dimensional space $\P_{k-1}$) we find $\delta>0$  such that
\[
\sup_{y\in A_{x,\eps},|y|\leq R_{k-1}}\frac{|P(y)|}{\varrho(y)}\leq 1 +\delta
\quad \Longrightarrow \quad
\frac{|\partial^{\alpha}P(0)|}{\alpha!}< \eps_\alpha
\]
for all $|\alpha|\leq k-1$ and $P\in \P_{k-1}$. 
From Lemma \ref{bound} we get $C_k>0$ such that, for all $P\in \P_k$, $x\in \partial K$, $\eps \in (0,1)$,
and $R>0$
\begin{equation}
\label{four4}
\sum_{|\alpha|=k}\frac{|\partial^{\alpha}P(0)|}{\alpha!}\leq \frac{C_k}{R}\sup_{y\in A_{x,\eps},|y|\leq R}\frac{|P(y)|}{R^{k-1}}.
\end{equation}
We choose $R_k>R_{k-1}$ such that
$$
\sfrac{C_k}{R_k}<\min\{\eps_\alpha:|\alpha|=k\}
\text{ and }
\sfrac{C_k}{R_k}< \sfrac{\delta\eps_0}{2R_{k-1}^{k}}.
$$
If now $P\in\P_k$ satisfies
$$
\sup_{y\in A_{x,\eps},|y|\leq R_{k}}\frac{|P(y)|}{\varrho(y)}\leq 1
$$
we apply (\ref{four4}) for $R=R_k$. Since $\varrho(y)$ is increasing we get for $|\alpha|=k$
$$
\frac{|\partial^{\alpha}P(0)|}{\alpha!}\leq\frac{C_k}{R_k}\sup_{y\in A_{x,\eps},|y|\leq R_k}\frac{|P(y)|}{R_{k}^{k-1}}<
\eps_\alpha\sup_{y\in A_{x,\eps},|y|\leq R_k}\frac{|P(y)|}{\varrho(y)}\leq\eps_\alpha.
$$

Let $P_k$ be the homogeneous part of $P$ of degree $k$.  Using $\varrho(y)\leq R_{k}^{k-1}$
for $|y|\leq R_k$ and again (\ref{four4}) for $R=R_k$  we obtain
$$
\sup_{y\in A_{x,\eps},|y|\leq R_{k-1}}\frac{|(P-P_k)(y)|}{\varrho(y)}\leq \sup_{y\in A_{x,\eps},|y|\leq R_{k}}\frac{|P(y)|}{\varrho(y)}+
\sup_{y\in A_{x,\eps},|y|\leq R_{k-1}}\frac{|P_k(y)|}{\varrho(0)}
$$
$$
\leq1+\frac{2}{\eps_0}\sum_{|\alpha|=k}\frac{|\partial^{\alpha}P(0)|}{\alpha!}R_{k-1}^{k}\leq 1 +
   \frac{2R_{k-1}^{k}}{\eps_0}\frac{C_k}{R_k}\sup_{y\in A_{x,\eps},|y|\leq R_{k}}\frac{|P(y)|}{R_{k}^{k-1}}
$$
$$
\leq  1+\delta \sup_{y\in A_{x,\eps},|y|\leq R_{k}}\frac{|P(y)|}{\varrho(y)} \leq 1+\delta.
$$
Using the induction assumption we conlude
$$
\frac{|\partial^{\alpha}P(0)|}{\alpha!}=\frac{|\partial^{\alpha}(P-P_k)(0)|}{\alpha!}< \eps_\alpha \text{ for }|\alpha|< k.
$$
\end{proof}

\begin{proof}[Proof of proposition \ref{momentproblem}]
Consider the space $\varphi\left(\N_0^d\right)$ as the locally convex direct sum of $\N_0^d$ copies of $\C$. It has the finest locally convex topology and
it is the dual of the product $\omega\left(\N_0^d\right)$ with the duality 
$\<(x_\beta)_{\beta \in\N_0^d},(y_\beta)_{\beta\in\N_0^d}\>=\sum_\beta x_\beta y_\beta$.
The set
$$
L=\{(\alpha!\delta_{\alpha,\beta})_{\beta\in\N_0^d}:\ \alpha\in\N_0^d\}\subseteq \omega\left(\N_0^d\right)
$$
is compact in $\omega\left(\N_0^d\right)$ and thus, the polar $L^{\circ}$ is a 0-neighbourhood in $\varphi\left(\N_0^d\right)$. Let $(\eps_{\alpha})_{\alpha\in\N_0^d}$ 
be a family of positive numbers such that $U=\{(\lambda_\alpha)_{\alpha\in\N_0^d}\in \varphi(\N_0^d):\ |\lambda_\alpha|<\eps_{\alpha}\}$ 
is a $0$-neighbourhood in $\varphi(\N_0^d)$ with $U\subseteq L^{\circ}$. For the weight function $\varrho$ from
lemma \ref{weight}, $x\in K$ and $\eps \in (0,1)$  we consider the spaces
$$
C_{x,\eps}:=\left\{f\in \CC(A_{x,\eps}):\lim_{y\to\infty}\frac{|f(y)|}{\varrho(y)}=0\right\},
$$
equipped with the weighted sup-norm $\|f\|_{x,\eps}$ whose unit ball is denoted by $B_{x,\eps}$. 
In a functional analytic form lemma \ref{weight} says that
the continuous linear mapping
$$
T_{x,\eps}: \varphi\left(\N_0^n\right) \rightarrow  C_{x,\eps},\;
(\lambda_\alpha)_{\alpha\in \N_0^d} \mapsto \left(y \mapsto \sum_{\alpha\in \N_0^d}\lambda_{\alpha}y^{\alpha}\right)
$$
satisfies $T_{x,\eps}^{-1}(B_{x,\eps})\subseteq U$ for each $x\in K$ and each $\eps \in (0,1)$.
Since the map $C_{x,\eps}\to \CC_0(A_{x,\eps})$, $f\mapsto f/\varrho$ is an isometry  Riesz's representation theorem yields
$$
C_{x,\eps}'=\left\{\frac{1}{\varrho}\mu:\ \mu\mbox{ a regular finite Borel measure on }A_{x,\eps}\right\},
$$
and the unit ball of $C_{x,\eps}'$ is
$$
D_{x,\eps}:=\left\{\frac{1}{\varrho}\mu:\ |\mu|(A_{x,\eps})\leq 1\right\}.
$$
Since $D_{x,\eps}$ is weak*-compact and $T_{x,\eps}^t$ is 
weak*-weak* continuous the bipolar theorem yields
$$
T_{x,\eps}^{t}(D_{x,\eps})=\overline{T_{x,\eps}^{t}(D_{x,\eps})} =
\overline{T_{x,\eps}^{t}(B_{x,\eps}^{\circ})}=T_{x,\eps}^{-1}(B_{x,\eps})^{\circ}\supseteq U^{\circ}\supseteq L^{\circ\circ}\supseteq L.
$$
Since $z=(\alpha! \delta_{\alpha,\beta})_{\beta\in\N_0^d} \in L$ we find a measure $\mu=\mu_{\alpha,x,\eps}$ on $A_{x,\eps}$ with total variation
bounded by $1$ such that $z= T^t_{x,\eps} (\mu/\varrho)$. This implies the assertion of proposition \ref{momentproblem} since for the canonical unit vectors
$e_\gamma=(\delta_{\gamma,\beta})_{\beta\in\N_0^d}$
\[
 \< T^t_{x,\eps} (\mu/\varrho) , e_\gamma\> = \mu/\varrho ( T_{x,\eps} e_\gamma) = \int y^\gamma/\varrho(y)\, d\mu(y). 
\]
\end{proof}

\bibliographystyle{amsalpha}
\bibliography{extension2-lit}
\end{document}